\documentclass[11pt]{article}
\usepackage{cite}
\usepackage{a4}
\usepackage{amsmath,amsfonts}
\usepackage{amssymb}
\usepackage{amsthm}
\usepackage{epic,gastex}

\newtheorem{prop}{Proposition}
\newtheorem{thm}{Theorem}
\newtheorem{lemma}[prop]{Lemma}
\newtheorem{cor}[prop]{Corollary}

\newcommand{\hm}[1]{#1\nobreak\discretionary{}{\hbox{\ensuremath{#1}}}{}}
\renewenvironment{proof}{\trivlist\item[\hskip\labelsep{\bf Proof}.]}{\qed\endtrivlist}
\newenvironment{proof*}{\trivlist\item[\hskip\labelsep{\bf Proof}]}{\qed\endtrivlist}
\newenvironment{proof1}{\trivlist\item[\hskip\labelsep{\bf Proof of Theorem~\ref{basis1}}.]}{\qed\endtrivlist}

\theoremstyle{definition}
\newtheorem{note}{Remark}

\newcommand{\VM}[1]{\textsc{Var-Memb}(#1)}

\DeclareMathOperator{\var}{var}
\DeclareMathOperator{\alf}{alph}

\begin{document}
\title{A minimal nonfinitely based semigroup
whose variety is polynomially recognizable}
\author{M.\,V.\,Volkov, S.\,V.\,Goldberg\\
\emph{Ural State University}\\
email: \{Mikhail.Volkov,Svetlana.Goldberg\}@usu.ru\\[.5em]
S.\,I.\,Kublanovsky\\
\emph{TPO ``Severny Ochag''}\\
email: stas@norths.spb.su}

\date{}

\maketitle

\begin{abstract}
We exhibit a 6-element semigroup that has no finite identity basis
but nevertheless generates a variety whose finite membership
problem admits a polynomial algorithm.
\end{abstract}

\renewcommand{\abstractname}{}

\begin{abstract}
\end{abstract}

\section{Motivation and overview}

Developments in the theory of computational complexity have shed
new light on algorithmic aspects of algebra. It has turned out
that many basic algorithmic questions whose decidability is well
known and/or obvious give rise to fascinating and sometimes very
hard problems if one looks for the computational complexity of
corresponding algorithms\footnote{In this paper complexity  is
understood in the sense of the monographs~\cite{GJ82,Pa94}; the
reader can find there the definitions of the complexity classes
\textsf{P}, \textsf{NP}, \textsf{EXPSPACE}, and \textsf{2-EXPTIME}
that are mentioned below.}. A good example is provided by the
problem \textsc{Var-Memb} studied in this paper: \emph{given two
finite algebras $A$ and $B$ of the same similarity type, decide
whether or not the algebra $A$ belongs to the variety generated by
the algebra $B$}. Clearly, the problem \textsc{Var-Memb} is of
importance for universal algebra in which equational
classification of algebras is known to play a central role. At the
same time, the problem is of interest in computer science and, in
particular, in formal specification theory
(cf.~\cite[Section~1]{BS00}) and in formal language theory (see
discussion below). The fact that the problem \textsc{Var-Memb} is
decidable easily follows from Tarski's HSP-theorem and has already
been mentioned in Kalicki's pioneering paper~\cite{Ka52}. The
question about computational complexity of this problem has been
explicitly posed much later, namely, in Kharlampovich and Sapir's
well-known survey, see~\cite[Problem~2.5]{KS95}. A systematic
study of this question has then started and brought interesting
and rather unexpected results.  Bergman and Slutzki~\cite{BS00}
extracted an upper bound from an analysis of Kalicki's proof: the
problem \textsc{Var-Memb} belongs to the class \textsf{2-EXPTIME}
of problems solvable in double exponential time.  For some time it
appeared that this bound was very loose but then
Szekely~\cite{Sz02} showed that the problem is \textsf{NP}-hard,
and Kozik~\cite{Ko04,Ko07} proved that it is even
\textsf{EXPSPACE}-hard. Finally, Kozik~\cite{Kozik} has shown that
the problem  \textsc{Var-Memb} is \textsf{2-EXPTIME}-complete,
thus confirming that the bound by Bergman and Slutzki in general
is tight. Thus, the problem \textsc{Var-Memb} has turned out to be
one of the hardest algorithmic problems of universal algebra.

The problem \textsc{Var-Memb} is of a special interest for
semigroups in the view of the well-known Eilenberg
correspondence~\cite{Ei76} between varieties of regular language
and pseudovarieties of semigroups\footnote{Recall that a
\emph{semigroup pseudovariety} is a class of finite semigroups
closed under taking subsemigroups and homomorphic images and under
forming finitary direct products. In particular, the class
$\mathcal{V}_\mathrm{fin}$ of all finite semigroups in a given
variety $\mathcal{V}$ is a  pseudovariety.}. By this
correspondence, a regular language belongs to some language
variety  $\mathbf{L}$ if and only if the syntactic semigroup of
the language belongs to the pseudovariety of semigroups
corresponding to $\mathbf{L}$. Therefore, estimating complexity of
the semigroup version of \textsc{Var-Memb}, one can deduce
conclusions about computational complexity  of some important
problems in formal language theory. At the present moment, the
precise complexity of the problem \textsc{Var-Memb} for semigroups
has not yet been determined but it is known that the problem is
\textsf{NP}-hard (Jackson and McKenzie~\cite{JM06}). In what
follows, we concentrate on the problem \textsc{Var-Memb} for
semigroups.

A reasonable strategy for analyzing \textsc{Var-Memb} in detail
consists in fixing the semigroup $B$ as a parameter so that the
role of an input is played by the semigroup $A$ only. This
approach splits \textsc{Var-Memb} into a series of problems
\VM{$B$} that are parameterized by finite semigroups and leads to
the question of classifying finite semigroups with respect to
computational complexity of the membership problem for the
varieties these semigroups generate. Let us proceed with precise
definitions.

Let $B$ be an arbitrary but fixed finite semigroup and let
$\var{B}$ be the variety generated by $B$. The problem \VM{$B$} is
a combinatorial decision problem whose instance is an arbitrary
finite semigroup $A$, and the answer to the instance $A$ is
``YES'' or ``NO'' depending on whether or not $A$ belongs to the
variety $\var{B}$. If there exist a deterministic Turing machine
and a polynomial $p(x)$ with integer coefficients, both depending
on the semigroup $B$ only, such that for every finite semigroup
$A$, the machine decides in time at most $p(|A|)$ whether or not
$A$ belongs to the variety $\var{B}$, then we say that $B$ is a
\emph{semigroup with polynomially recognizable variety}.
Similarly, if there is no such polynomial, but there exists a real
constant $\alpha>1$ such that for every finite semigroup $A$, the
machine decides the same question in time at most $\alpha^{|A|}$,
then we say that $B$ is a \emph{semigroup with exponentially
recognizable variety}, etc. The classification question mentioned
in the previous paragraph is essentially the question of an
efficient characterization of finite semigroups with polynomially
(exponentially etc.) recognizable varieties. We notice that
Jackson and McKenzie~\cite{JM06} have exhibited a 56-element
semigroup $J{\kern-1pt}M$ for which the problem
\VM{$J{\kern-1pt}M$} is \textsf{NP}-hard. This means that under
the standard assumption $\mathsf{P}\hm\ne\mathsf{NP}$, the
semigroup $J{\kern-1pt}M$ is not a semigroup with polynomially
recognizable variety.

Semigroup with polynomially recognizable varieties could be
alternatively called semigroups with easily verifiable identities.
Indeed, by the definition the variety $\var{B}$ consists of all
semigroups satisfying every identity holding in $B$, whence
testing membership of a given semigroup $A$ in the variety
$\var{B}$ is nothing but testing whether $A$ satisfies every
identity  of the semigroup $B$. This observation immediately
implies a simple but important conclusion:
\begin{lemma}
\label{lemma 1.1} If all identities holding in a semigroup $B$
follow from a finite set $\Sigma$ of such identities, then $B$
generates a polynomially recognizable variety.
\end{lemma}

\begin{proof}
Under the premise of the lemma, in order to check whether or not a
given finite semigroup $A$ belongs to the variety $\var{B}$, it
suffices to check whether or not $A$ satisfies all identities in
$\Sigma$. To check that an identity $u=v$ in $\Sigma$ holds in
$A$, provided that $u$ and $v$ together depend on $m$ variables,
one can simply substitute for the variables all possible
$m$-tuples of elements of $A$ and then check whether or not all
substitutions yield equal values to the words $u$ and $v$. The
number of $m$-tuples subject to the evaluation is $|A|^m$ while
time needed to calculating the values of the words $u$ and $v$
depends only on the length of these words and not on the size of
the semigroup $A$. Hence the total time consumed by this algorithm
is bounded by a polynomial of degree $m$ in $|A|$. Since the
number of identities in $\Sigma$ also does not depend on the size
of $A$, we see that the inclusion $A\in\var{B}$ can be verified in
polynomial in $|A|$ time.
\end{proof}

A semigroup that satisfies the premise of Lemma~\ref{lemma 1.1} is
said to be \emph{finitely based}. The question which finite
semigroups are finitely based and which are not plays a central
role in the theory of semigroup varieties for more than 40~year,
see~\cite{Vo01} for a survey of the corresponding area.
Lemma~\ref{lemma 1.1} links this question and the problem of
characterizing finite semigroups with polynomially recognizable
varieties.

It is easy to see that in general a polynomially recognizable
variety need not be finitely based. Here the variety
$\mathcal{B}_4\mathcal{B}_2$ of all extensions of groups of
exponent~4 by groups of exponent~2 studied by Kleiman~\cite{Kl73}
can serve as a simple example. (Since this class consists of
periodic groups, it also forms a semigroup variety.) Indeed, it is
shown in~\cite{Kl73} that the variety $\mathcal{B}_4\mathcal{B}_2$
is non\-finitely based. On the other hand, if $A$ is a finite
semigroup, then in order to test the membership of $A$ in
$\mathcal{B}_4\mathcal{B}_2$, it suffices to test whether or not
$A$ is a group, and if this is the case, to check whether or not
the normal subgroup generated by all squares in $A$ has
exponent~4. Clearly, both these checks can be performed in
polynomial (in fact, cubic) in $|A|$ time.

The situation changes essentially if one considers a variety
generated by a finite semigroup. Here one cannot find a similar
example among varieties consisting only of groups because by a
classic result by Oates and Powell~\cite{OP64} every finite group
is finitely based. In~\cite[Theorem~3.53]{KS95} the authors
describe a certain semigroup variety $\mathcal{S}$ and claim that
$\mathcal{S}$ is polynomially recognizable and that one can deduce
from Sapir's result~\cite{Sa91} that $\mathcal{S}$ is
non\-finitely based and is generated by a finite semigroup.
However, an algorithm for testing membership of a finite semigroup
in the variety $\mathcal{S}$ is described in neither~\cite{KS95}
nor subsequent publications; no finite semigroup generating
$\mathcal{S}$ is explicitly exhibited. Moreover, the reference
to~\cite{Sa91} does not appear to be fully legitimate because it
is clear from the description of the variety $\mathcal{S}$
in~\cite[Theorem~3.53]{KS95} that all groups in $\mathcal{S}$ have
exponent~4 while semigroup varieties analyzed in~\cite{Sa91}
contain only groups of odd exponent. We do not doubt the validity
of the claim announced in~\cite[Theorem~3.53]{KS95}, but we
believe nevertheless that in order to clarify the relationship
between the properties of a finite semigroup ``to be finitely
based'' and ``to generate a polynomially recognizable variety'',
it is worthwhile to provide a more transparent example with
complete justification. This is the goal of the present paper.

We exhibit a 6-element semigroup $A{\kern-1pt}C_2$ that is
non\-finitely based and at the same time generates a polynomially
recognizable variety. We explicitly write down an infinite
identity basis for $A{\kern-1pt}C_2$ and describe in detail a
polynomial algorithm  for testing membership of an arbitrary
finite semigroup in the variety $\var A{\kern-1pt}C_2$.

We notice that our example has the minimum possible number of
elements because it is well known that every semigroup with five
or fewer elements is finitely based~\cite{Tr83,Tr91}.
Surprisingly, it seems that the semigroup $A{\kern-1pt}C_2$ has
not yet appeared in the literature. The reader may be aware of the
other 6-element non\-finitely based semigroup, the so-called
\emph{Brandt monoid} $B_2^1$ formed by the following $2\times
2$-matrices
$$ \begin{pmatrix}
0 & 0 \\ 0 & 0
\end{pmatrix}, ~~
\begin{pmatrix}
1 & 0 \\ 0 & 1
\end{pmatrix}, ~~
\begin{pmatrix}
1 & 0 \\ 0 & 0
\end{pmatrix}, ~~
\begin{pmatrix}
0 & 1 \\ 0 & 0
\end{pmatrix}, ~~
\begin{pmatrix}
0 & 0 \\ 1 & 0
\end{pmatrix}, ~~
\begin{pmatrix}
0 & 0 \\ 0 & 1
\end{pmatrix}$$
under usual matrix multiplication. Since the pioneering paper by
Perkins~\cite{Pe69}, the Brandt monoid appears over and over again
in publications on the theory of semigroup varieties for more than
40~years. It is known that $B_2^1$ has many remarkable properties
(including those related to computational complexity,
see~\cite{Se05,Klima}) but the question about the complexity of
the problem \VM{$B_2^1$} still remains open (and is very
intriguing in our opinion). Therefore at the moment one cannot use
the Brandt monoid as the example we are looking for.

A further interesting property of the semigroup $A{\kern-1pt}C_2$
is that $\var A{\kern-1pt}C_2$ is a \emph{limit} variety, that is,
a minimal (under class inclusion) non\-finitely based variety,
see~\cite{LeeVolkov}. Thus, our example is minimal not only with
respect to the number of elements but also with respect to the
natural ordering of varieties.

The paper is structured as follows. In Section~2 we construct the
semigroup $A{\kern-1pt}C_2$, establish its identity basis, and
give a structural characterization of semigroups in the variety
$\var A{\kern-1pt}C_2$. In Section~3 we show how to use this
characterization in order to check, given a finite semigroup $S$,
whether or not $S\in\var A{\kern-1pt}C_2$ in time $O(|S|^3)$.

We assume the reader's acquaintance with rudiments of semigroup
theory up to the first three chapters of the monograph~\cite{ClPr}
as well as with some basics of the theory of varieties,
see~\cite[Chapter~II]{BuSa81}. For the reader's convenience we
recall here the notion of a Rees matrix semigroup which is
important for the present paper.

Let $G$ be a group, 0 be a symbol not in $G$. Further, let
$I,\Lambda$ be non-empty sets and $P=(p_{\lambda,i})$ be a
$\Lambda\times I$-matrix over $G\cup\{0\}$. The \emph{Rees matrix
semigroup $M^0(G; I, \Lambda; P)$ over the group $G$ with the
sandwich-matrix $P$} is the semigroup on the set $(I\times G\times
\Lambda)\cup \{0\}$ under multiplication
\begin{gather*}
x\cdot 0=0\cdot x=0\ \text{ for all $x\in (I\times G\times \Lambda)\cup \{0\}$},\\
(i,g,\lambda)\cdot(j,h,\mu)=\begin{cases}0&\mbox{ if }p_{\lambda,j}=0,\\
(i,gp_{\lambda,j}h,\mu)&\mbox{ if }p_{\lambda,j}\ne0.
\end{cases}
\end{gather*}

\section{The semigroup $A{\kern-1pt}C_2$ and its identity basis}

Let $A_2$ denote the 5-element idempotent-generated 0-simple
semigroup. It can be defined in the class of semigroups with zero
by the following presentation:
$$A_2=\langle a,b\mid a^2=aba=a,\ bab=b,\ b^2=0\rangle=\{a,b,ab,ba,0\}.$$
The semigroup $A_2$ can be also thought of as the semigroup formed
by the following $2\times2$-matrices (over an arbitrary field)
$$\begin{pmatrix}
0 & 0\\ 0 & 0
\end{pmatrix},\
\begin{pmatrix}
1 & 0\\ 0 & 0
\end{pmatrix},\
\begin{pmatrix}
0 & 1\\ 0 & 0
\end{pmatrix},\
\begin{pmatrix}
1 & 0\\ 1 & 0
\end{pmatrix},\
\begin{pmatrix}
0 & 1\\ 0 & 1
\end{pmatrix}$$
under the usual multiplication of matrices or as the Rees matrix
semigroup over the trivial group $E=\{1\}$ with the
sandwich-matrix $\left(\begin{smallmatrix}1 & 1\\ 0 &
1\end{smallmatrix}\right)$.

The semigroup $A{\kern-1pt}C_2$ is obtained by adding a new
element $c$ to the semigroup $A_2$. The multiplication in
$A{\kern-1pt}C_2$ extends the multiplication in $A_2$ and the
products involving the added element $c$ are defined as follows:
$$ c^2 = 0\ \text{ and }\ xc = cx = c\  \text{ for all $x\in A_2$}.$$
(In order to avoid any confusion, we stress that the element 0 is
no longer a zero in $A{\kern-1pt}C_2$ since $0c=c0=c$.) The fact
that the multiplication defined this way is associative can be
straightforwardly verified but can be also obtained without
calculations if one observes that the groupoid $A{\kern-1pt}C_2$
is isomorphic to a subsemigroup of the direct product of the
semigroup $A_2$ with the cyclic group $C_2=\langle c\mid
c^2=1\rangle=\{c,1\}$, namely, to the subsemigroup consisting of
all the pairs of the form $(x,1)$, where $x\in A_2$, and the pair
$(0,c)$.

By the construction, $A_2$ is a subsemigroup in $A{\kern-1pt}C_2$;
on the other hand, the elements 0 and $c$ form in
$A{\kern-1pt}C_2$ a subgroup isomorphic to the group $C_2$. (Thus,
$A{\kern-1pt}C_2$ is obtained via amalgamating $A_2$ and $C_2$
such that the zero of the semigroup $A_2$ is identified with the
identity element of the group $C_2$.) Since $A_2,C_2\in\var
A{\kern-1pt}C_2$, we have $A_2\times C_2\in\var A{\kern-1pt}C_2$.
Conversely, as mentioned above, the semigroup $A{\kern-1pt}C_2$
embeds into the direct product $A_2\times C_2$ whence
$A{\kern-1pt}C_2\in\var(A_2\times C_2)$. We see that the
semigroups $A{\kern-1pt}C_2$ and $A_2\times C_2$ generate the same
variety, in other words, they satisfy the same identities. It
follows from the results of~\cite{Vo89} (cf.\ Remark~2 in the
discussion of the main theorem there) that for every group $G$ of
finite exponent the direct product $A_2\times G$ is non\-finitely
based. Hence we obtain the first property of the semigroup
$A{\kern-1pt}C_2$ we need.

\begin{lemma}
\label{nfb} The semigroup $A{\kern-1pt}C_2$ is non\-finitely
based.
\end{lemma}

\begin{note}
The short note~\cite{Ma83} contains an announcement (with a proof
sketch) of the following fact: the Rees matrix semigroup over the
group $C_2$ with the sandwich-matrix $\left(\begin{smallmatrix}1 & 1\\
0 & 1\end{smallmatrix}\right)$ is non\-finitely based. It is not
hard to show that this 9-element semigroup generates the same
variety as the semigroups $A{\kern-1pt}C_2$ and $A_2\times C_2$.
Therefore Lemma~\ref{nfb} can also be deduced from the result
of~\cite{Ma83}.
\end{note}

Now we describe the identities of the semigroup $A{\kern-1pt}C_2$.
For a word $w$, we denote by $\alf(w)$ the set of variables that
occur in $w$ and by $|w_x|$ the number of occurrences of the
variable $x$ in $w$. Given a word $w$, we assign to it a directed
graph $G(w)$ whose vertex set is $\alf(w)$ and whose edges
correspond to factors of length~2 in $w$ as follows: $G(w)$ has a
directed edge from $x$ to $y$ ($x,y\in\alf(w)$) if and only if
some occurrence of $x$ in $w$ immediately precedes some occurrence
of $y$.
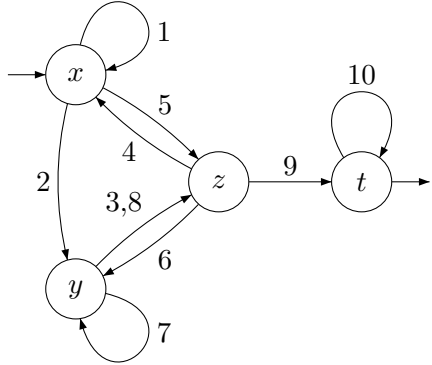
\begin{figure}[ht]
\begin{center}
\unitlength .95mm
\begin{picture}(40,50)(0,3)
\thinlines
\node[Nmarks=i](A)(0,40){$x$}
\node(B)(0,10){$y$}
\node(C)(20,25){$z$}
\node[Nmarks=f](D)(40,25){$t$}
\drawedge[curvedepth=-2.5,ELside=r](A,B){2}
\drawedge(C,D){9}
\gasset{curvedepth=1.5}
\drawloop[ELpos=65,loopangle=45](A){1}
\drawedge(B,C){3,8}
\drawedge(C,A){4}
\drawedge(A,C){5}
\drawedge(C,B){6}
\drawloop[ELpos=35,loopangle=-45](B){7}
\drawloop(D){10}
\end{picture}
\end{center}
\caption{The graph of the word $w=x^2yzxzy^2zt^2$ and the
corresponding walk}\label{example1}
\end{figure}
We will distinguish two (not necessarily different) vertices in
$G(w)$: the \emph{initial vertex}, that is the first letter of
$w$, and the \emph{final vertex}, that is the last letter of $w$.
Then the word $w$ defines a walk through the graph $G(w)$ that
starts at the initial vertex, ends at the final vertex and
traverses each edge of $G(w)$ (some of the edges can be traversed
more than once).

Figure~\ref{example1} shows the graph $G(w)$ for the word $w=
x^2yzxzy^2zt^2$. The ingoing and the outgoing marks show
respectively the initial and the final vertices of the graph. In
Fig.\,\ref{example1} each edge of the graph is labelled by the
number[s] corresponding to the occurrence[s] of the edge in the
walk induced by the word~$w$. We stress that, in contrast to the
vertex names and the ingoing/outgoing marks, these labels are
\textbf{not} considered as a part of the data making the graph
$G(w)$. Therefore the graph does not determine the word $w$: for
instance, as the reader can easily check, the word
$xy^3zyzx^2zyzt^3$ has exactly the same graph (but corresponds to
a different walk through it, see Fig.\,\ref{example2}).
\begin{figure}[ht]
\begin{center}
\unitlength .95mm
\begin{picture}(40,50)(0,3)
\thinlines
\node[Nmarks=i](A)(0,40){$x$}
\node(B)(0,10){$y$}
\node(C)(20,25){$z$}
\node[Nmarks=f](D)(40,25){$t$}
\drawedge[curvedepth=-2.5,ELside=r](A,B){1}
\drawedge(C,D){12}
\gasset{curvedepth=1.5}
\drawloop[ELpos=65,loopangle=45](A){8}
\drawedge(B,C){4,6,11}
\drawedge(C,A){7}
\drawedge(A,C){9}
\drawedge(C,B){5,10}
\drawloop[ELpos=35,loopangle=-45](B){2,3}
\drawloop(D){13,14}
\end{picture}
\end{center}
\caption{Another walk through the graph of
Fig.\,\ref{example1}}\label{example2}
\end{figure}
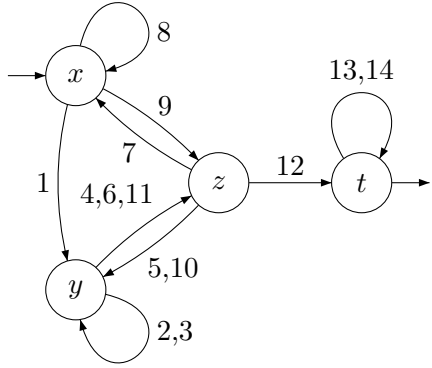

Observe that in terms of the graph $G(w)$, the number $|w_x|$
represents the number of times that the walk induced by the
word~$w$ visits the vertex $x$.

\begin{prop}
\label{identities of AC2} An identity $u=v$ holds true in the
semigroup $A{\kern-1pt}C_2$ if and only if the graphs $G(u)$ and
$G(v)$ coincide and, for each variable~$x$, the numbers $|u_x|$
and $|v_x|$ have the same parity.
\end{prop}

\begin{proof}
We have mentioned above that the semigroups $A{\kern-1pt}C_2$ and
$A_2\times C_2$ satisfy the same identities. Clearly, an identity
holds in the semigroup $A_2\times C_2$ if and only if it holds in
each of the semigroups $A_2$ and $C_2$. It is known that an
identity $u=v$ holds true in the semigroup $A_2$ if and only if
the graphs $G(u)$ and $G(v)$ coincide, see \cite{Tr81}\footnote{In
the literature (see, for instance, \cite{Lee04} or~\cite{Tr94})
one sometimes refers to~\cite{Mash78} as the source for this
result even though the paper~\cite{Mash78} does not deal with the
semigroup $A_2$ at all. Apparently, this mistake originates from
an erroneous reference in the survey paper~\cite{ShVo85}.}.
Further, it is known (and easy to verify) that an identity $u=v$
holds true in the group $C_2$ if and only if the numbers $|u_x|$
and $|v_x|$ have the same parity for each variable~$x$.
\end{proof}

Proposition~\ref{identities of AC2} immediately implies
\begin{cor}
\label{basis} The identities
\begin{gather}
\label{eq:1} x^2 = x^4,\\
\label{eq:2} xyx = (xy)^3x,\\
\label{eq:3} xyxzx = xzxyx,\\
\label{eq:4} (x_1^2x_2^2\cdots x_n^2)^2 = (x_1^2x_2^2\cdots x_n^2)^3,\quad n=2,3,\dotsc,
\end{gather}
hold true in the semigroup $A{\kern-1pt}C_2$.
\end{cor}

\begin{proof}
It is easy to see that for each of the identities
\eqref{eq:1}--\eqref{eq:4}, the graph of its left hand side
coincides with the graph of its right hand side (the graphs are
shown in Fig.\,\ref{4 graphs}) and each variable
\begin{figure}[ht]
\begin{center}
\unitlength .95mm
\begin{picture}(120,50)
\thinlines
\node[Nmarks=if](A)(0,40){$x$}
\drawloop(A){}
\put(-2,30){\eqref{eq:1}}
\node[Nmarks=if,iangle=90,fangle=-90](B)(30,40){$x$}
\node(C)(50,40){$y$}
\drawedge[curvedepth=2](B,C){}
\drawedge[curvedepth=2](C,B){}
\put(38,30){\eqref{eq:2}}
\node[Nmarks=if,iangle=90,fangle=-90](D)(100,40){$x$}
\node(E)(80,40){$y$}
\node(F)(120,40){$z$}
\drawedge[curvedepth=2](D,E){}
\drawedge[curvedepth=2](E,D){}
\drawedge[curvedepth=2](D,F){}
\drawedge[curvedepth=2](F,D){}
\put(92,30){\eqref{eq:3}}
\node[Nmarks=i](A1)(20,10){$x_1$}
\node(A2)(40,10){$x_2$}
\node(A3)(80,10){\small$x_{n{-}1}$}
\node[Nmarks=f](A4)(100,10){$x_n$}
\drawloop(A1){}
\drawloop(A2){}
\drawloop(A3){}
\drawloop(A4){}
\drawedge(A1,A2){}
\drawedge(A3,A4){}
\drawedge[curvedepth=8](A4,A1){}
\put(58,-4){\eqref{eq:4}}
\node[Nframe=n](G)(60,10){$\cdots$}
\drawedge(A2,G){}
\drawedge(G,A3){}
\end{picture}
\end{center}
\caption{The graphs of the identities
\eqref{eq:1}--\eqref{eq:4}}\label{4 graphs}
\end{figure}
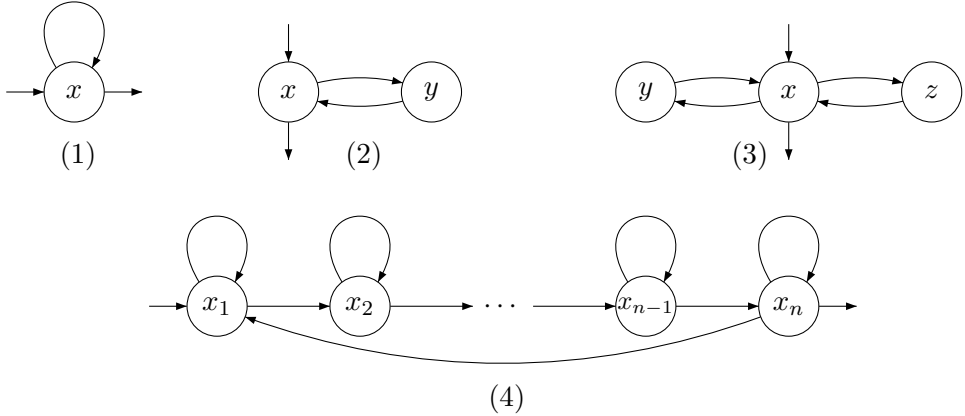
occurs on the left and on the right with the same parity.
\end{proof}

We aim to clarify the structural meaning of the identities
\eqref{eq:1}--\eqref{eq:4}. We start with the series~\eqref{eq:4}.
Recall that a semigroup is said to be \emph{combinatorial} if all
of its subgroups are singletons.

\begin{prop}
\label{aperiodic core} If a semigroup $S$ satisfies the identities
\eqref{eq:4}, then the subsemigroup generated by all idempotents
of $S$ is combinatorial. If $S$ satisfies the identity
\eqref{eq:1}, then the converse is true as well.
\end{prop}

\begin{proof}
Since every idempotent can be represented as a square, products of
the form $x_1^2 \cdots x_n^2$, $n=2,3,\dotsc$, represent all
elements of the subsemigroup $T$ generated by all idempotents of
$S$. If $S$ satisfies~\eqref{eq:4}, then $T$ satisfies the
identity
\begin{equation}
\label{eq:5} x^2 = x^3
\end{equation}
that cannot hold in a non-singleton group. Hence, the subsemigroup
$T$ is combinatorial.

Conversely, let $S$ satisfies the identity \eqref{eq:1}. Then the
subsemigroup $T$ also satisfies this identity but in a
combinatorial semigroup \eqref{eq:1} implies~\eqref{eq:5}. In the
presence of the identity \eqref{eq:1}, the square of each element
of $S$ is an idempotent whence the values of all products of the
form $x_1^2 \cdots x_n^2$, $n=2,3,\dotsc$, lie in $T$.
Substituting these products for the variable in~\eqref{eq:5}, we
see that $S$ satisfies all identities from the
series~\eqref{eq:4}.
\end{proof}

The variety generated by all completely 0-simple semigroups whose
subgroups have exponent $n$ is commonly denoted by
$\mathcal{R{\kern-1pt}S}_n$. Clearly, the semigroup
$A{\kern-1pt}C_2$ belongs to the variety
$\mathcal{R{\kern-1pt}S}_2$. The next results reveals the role of
the identities \eqref{eq:1}--\eqref{eq:3}:
\begin{prop}
\label{RS2} The identities \eqref{eq:1}--\eqref{eq:3} form an
identity basis of the variety $\mathcal{R{\kern-1pt}S}_2$.
\end{prop}

We do not prove Proposition~\ref{RS2} here because it is not used
in the present paper. We notice that various identity bases for
$\mathcal{R{\kern-1pt}S}_n$ have been provided
in~\cite{Mash91,Mash97,Ha97}\footnote{However, the identity basis
for $\mathcal{R{\kern-1pt}S}_2$ specified in Proposition~\ref{RS2}
is not a specialization of the bases for
$\mathcal{R{\kern-1pt}S}_n$  provided
in~\cite{Mash91,Mash97,Ha97}.}. Unfortunately, the proofs of the
corresponding results in these papers cannot be considered as
complete because they all essentially use a lemma
from~\cite{Mash91} whose proof in~\cite{Mash91} is wrong. We shall
discuss these nuances in the course of the proof of the next
theorem that plays a key role in the present paper.

\begin{thm}
\label{basis1} The identities \eqref{eq:1}--\eqref{eq:4} form an
identity basis for $A{\kern-1pt}C_2$.
\end{thm}

The proof follows a scheme suggested in~\cite{LeeVolkov06}. We
need a few auxiliary statements.

A word $w$ of length at least~2 is said to be \emph{connected} if
its graph $G(w)$ is strongly connected\footnote{This concept
sometimes appears in the literature under different names. For
instance, in~\cite{Mash91} a word $w$ of length at least~2 is said
to be \emph{covered by cycles} each if each of its factors of
length 2 occurs in a factor of $w$that begins and ends with a
common letter. In the language of the graph $G(w)$, this means
that each directed edge $x\to y$ of $G(w)$ belongs to a directed
cycle (namely, to the walk induced by a factor of $w$ that starts
and ends with the same letter and contains $xy$). It is one of the
basic facts of the theory of directed graphs (cf.~\cite{Ore62},
Theorem~8.1.5) that such a graph is strongly connected if and only
if each of its directed edges belongs to a directed cycle. Thus,
words covered by cycles in the sense of~\cite{Mash91} are
precisely connected words in our sense. Yet another name for an
obviously equivalent concept has been introduced
in~\cite{Pollak-02}, where a word $w$ of length at least~2 is said
to be \emph{prime} if it cannot be decomposed as $w=w'w''$ with
$\alf(w')\cap\alf(w'')=\varnothing$.}. Let $\mathcal{V}$ be the
variety defined by the identities \eqref{eq:1}--\eqref{eq:4}.

\begin{lemma}
\label{regular elements} If $w$ is a connected word and $S$ is a
semigroup in $\mathcal{V}$, then every value of $w$ in $S$ is a
regular element in~$S$.
\end{lemma}

\begin{proof}
We recall that an element $s\in S$ is said to be \emph{regular in
$S$} if there exists an element $s^\prime \in S$ such that $ss's =
s$. Therefore in order to prove the lemma it suffices to construct
a word $w'$ such that the variety $\mathcal{V}$ satisfies the
identity $w = ww'w$. If the word $w$ begins and ends with the same
variable, then we can apply the identity~\eqref{eq:2} to it (or
the identity~\eqref{eq:1} in the case when $w$ is the square of a
variable) and we immediately get the necessary conclusion. We may
therefore assume that $w$ begins with a variable $x$ and ends with
a variable $y$ such that $x\ne y$.

Since the word $w$ is connected, each of the variables $x$ and $y$
occurs in $w$ more than once. We want to show that, applying the
identities \eqref{eq:2} and \eqref{eq:3}, one can transform $w$
into a word in which some occurrence of the variable $x$ appears
after some occurrence of the variable $y$. For this, it is
convenient to prove a slightly more general lemma.

\begin{lemma}
\label{insteadmash} Let $w$ be a connected word, $x,y\in\alf(w)$
and
\begin{equation}
\label{representation} w=w_1xw_2yw_3,\  \text{ причем } \
x\notin\alf(w_2yw_3) \text{ and } y\notin\alf(w_1xw_2).
\end{equation}
Applying the identities \eqref{eq:2} and \eqref{eq:3}, one can
transform $w$ into a word $w_1xw'_2yw_3$ such that
$x,y\in\alf(w'_2)$ and some occurrence of the variable $x$ in
$w'_2$ appears after some occurrence of the variable $y$ in $w'_2$
\end{lemma}

\begin{proof}
First of all, we observe that if some occurrences of the variables
$x$ and $y$ happen between two occurrences of some variable $z$,
then the desired transformation can be achieved by an application
of the identity \eqref{eq:2} to the factor bordered by these two
occurrences of $z$:
\begin{align*}
w&=\underbrace{w_{11}zw_{12}}_{w_1}xw_2y\underbrace{w_{31}zw_{32}}_{w_3}=w_{11}(zw_{12}xw_2yw_{31})^3zw_{32} &\text{(by \eqref{eq:2})}\\
 &=\underbrace{w_{11}zw_{12}}_{w_1}xw_2yw_{31}zw_{12}xw_2\underline{\underline{y}}
   w_{31}zw_{12}\underline{\underline{x}}w_2y\underbrace{w_{31}zw_{32}}_{w_3}.&
\end{align*}
(The ``permuted'' occurrences of the variables $x$ and $y$ are
underlined twice.)

Now we induct on the length of the word $w_2$ in the decomposition
\eqref{representation}, that is, on the distance between the right
most occurrence of $x$ and the left most occurrence of $y$. If
this distance is equal to 0, then the word $w$ has $xy$ as a
factor. Since $w$ is connected, this factor should appear between
two occurrences of some variable $z$, and then the argument from
the previous paragraph applies. This proves the induction basis.

Now suppose that in the decomposition \eqref{representation} the
word $w_2$ is not empty. In view of the first paragraph of the
proof, we can assume that $\alf(w_1x)\cap\alf(yw_3)=\varnothing$.
Since $w$ is connected, the word $w_2$ must have common variables
with each of the words $w_1$ and $w_3$. Consider two cases.

\smallskip

\noindent\textbf{\emph{Case 1}.} \emph{Some variable
$z\in\alf(w_1)\cap\alf(w_2)$ occurs in the word $w_2$ to the left
of some variable $t\in\alf(w_2)\cap\alf(w_3)$}.

In this case, the desired transformations are as follows:
\begin{align*}
w&=\underbrace{w_{11}zw_{12}}_{w_1}x\underbrace{w_{21}tw_{22}zw_{23}}_{w_2}y\underbrace{w_{31}tw_{32}}_{w_3}&\\
 &=w_{11}(zw_{12}xw_{21}tw_{22}z)^3w_{23}yw_{31}tw_{32}&\text{(by \eqref{eq:2})}\\
 &=w_{11}zw_{12}xw_{21}tw_{22}zw_{12}xw_{21}tw_{22}zw_{12}xw_{21}tw_{22}zw_{23}yw_{31}tw_{32}&\\
 &=w_{11}zw_{12}xw_{21}tw_{22}zw_{12}xw_{21}tw_{22}zw_{12}xw_{21}(tw_{22}zw_{23}yw_{31})^3tw_{32}
                                                       &\text{(by \eqref{eq:2})}\\
 &=w_{11}zw_{12}xw_{21}tw_{22}zw_{12}xw_{21}\underline{tw_{22}zw_{12}xw_{21}tw_{22}zw_{23}yw_{31}t}\times{}&\\
 &\phantom{=}\times w_{22}zw_{23}yw_{31}tw_{22}zw_{23}yw_{31}tw_{32}&\\
 &=\underbrace{w_{11}zw_{12}}_{w_1}xw_{21}tw_{22}zw_{12}xw_{21}tw_{22}zw_{23}
   \underline{\underline{y}}w_{31}tw_{22}zw_{12}\underline{\underline{x}}w_{21}t\times{}&\\
 &\phantom{=}\times w_{22}zw_{23}yw_{31}tw_{22}zw_{23}y\underbrace{w_{31}tw_{32}}_{w_3}
                                                       &\text{(by \eqref{eq:3}).}
\end{align*}
(The factor to which the identity \eqref{eq:3} has been applied is
underlined once while the ``permuted'' occurrences of the
variables $x$ and $y$ are underlined twice.)

\smallskip

\noindent\textbf{\emph{Case 2}.} \emph{Every variable from
$\alf(w_1)\cap\alf(w_2)$ occurs in the word $w_2$ to the right of
every variable from $\alf(w_2)\cap\alf(w_3)$}.

We take some variables $z\in\alf(w_1)\cap\alf(w_2)$ and
$t\in\alf(w_2)\cap\alf(w_3)$. Since both $z$ and $t$ occur in the
word $w_2$ while both $x$ and $y$ do not, the distance between the
right most occurrence of the variable $z$ and the left most
occurrence of the variable $t$ is less than distance between the
right most occurrence of the variable $x$ and the left most
occurrence of the variable $y$. Thus, we can apply the induction
assumption to the word $w$ and the variables $z,t$. This means
that if we write the word $w$ as
$$w=v_1zv_2tv_3,\  \text{ где } \ z\notin\alf(v_2tv_3) \text{ and } t\notin\alf(v_1zv_2),$$
then applying the identities \eqref{eq:2} and \eqref{eq:3}, we can
transform $w$ into a word $v=v_1zv'_2tv_3$ such that
$z,t\in\alf(v'_2)$ and some occurrence of the variable $z$ in
$v'_2$ appears after some occurrence of the variable $t$ in
$v'_2$. However  the word $v$ and the initial variables $x,y$ then
satisfy the condition of Case~1 that is considered above.
\end{proof}

Now we return to the proof of Lemma~\ref{regular elements}. Recall
that we consider a connected word $w$ that begins with a variable
$x$ and ends with a variable $y$ such that $x\ne y$. By
Lemma~\ref{insteadmash} we may assume that some occurrence of the
variable $x$ in $w$ appears after some occurrence of the variable
$y$ in $w$. Hence
\begin{align*}
w&=xw_1yw_2xw_3y=(xw_1yw_2)^3xw_3y &\text{(by \eqref{eq:2})}\\
 &=xw_1yw_2xw_1yw_2xw_1yw_2xw_3y&\\
 &=xw_1yw_2xw_1yw_2xw_1(yw_2xw_3)^3y &\text{(by \eqref{eq:2})}\\
 &=xw_1yw_2xw_1\underline{yw_2xw_1yw_2xw_3y}w_2xw_3yw_2xw_3y&\\
 &=xw_1\underline{yw_2xw_1yw_2xw_3y}w_2xw_1yw_2xw_3yw_2xw_3y &\text{(by \eqref{eq:3})}\\
 &=xw_1yw_2xw_3yw_2xw_3\underline{yw_2xw_1yw_2xw_3y}w_2xw_3y &\text{(by \eqref{eq:3})}\\
 &=\underbrace{xw_1yw_2xw_3y}_{w}\underbrace{w_2xw_1yw_2xw_3yw_2}_{w'}\underbrace{xw_1yw_2xw_3y}_{w} &\text{(by \eqref{eq:3}).}
\end{align*}
(The factors to which the identity \eqref{eq:3} has been applied
are underlined.) Thus, we have deduced an identity of the form
вида $w=ww'w$ from \eqref{eq:1}--\eqref{eq:3}, as required.
\end{proof}

\begin{note}
Lemma~\ref{regular elements} is a partial case of a similar result
claimed by Mashevitsky in~\cite[Lemma~6]{Mash91}, see
also~\cite[Lemma~7]{Mash94}. As we  have already mentioned, this
result has been used (with reference to~\cite{Mash91}) in several
important papers, in particular, \cite{Ha97} and~\cite{Mash97}.
However, its proof in~\cite{Mash91} contains a fatal flaw and so
does the English translation of the proof published
in~\cite{Mash94}. Namely, in~\cite{Mash91} Lemma~6 is deduced from
Lemma~5 which claims that every word $u$ covered by cycles can be
transformed modulo certain identities into a word of the form
$z_1u_1z_1\cdots z_ku_kz_k$ where $z_1,\dots,z_k$ are variables
and $z_{i+1}\in\alf(u_i)$ for all $i=1,\dots,k-1$ provided that
$k>1$. In order to justify the latter claim, Mashevitsky inducts
on $|\alf(u)|$ but in the course of the proof he illegitimately
applies the induction assumption to a factor that in general is
not covered by its cycles. The word $xyxzy$ can be used as a
concrete counter example showing that the argument
from~\cite{Mash91} does not work: here the induction assumption
should have been applied to the factor $zy$ which is certainly not
covered by its cycles.

We observe that our proof of Lemma~\ref{regular elements} invokes
only the identities \eqref{eq:1}--\eqref{eq:3}. Some modification
of our argument applies also to the identities considered
in~\cite{Mash91} and allows one to prove Lemma~6 of~\cite{Mash91}.
Thus, results of \cite{Ha97} and~\cite{Mash97} that rely on the
lemma are correct. Moreover, the third author of the present paper
has recently proved that already the identities \eqref{eq:1} and
\eqref{eq:2} suffice to ensure that the value of every connected
word is regular; an analogous generalization also holds in the
situation considered in~\cite{Mash91}.
\end{note}

A semigroup $S$ is called \emph{$E$-separable} if for every pair
$p,q$ of distinct elements in $S$, there exist idempotents $e,f\in
S$ such that $pe\neq qe$ and $fp\neq fq$.
\begin{lemma}
\label{separability} The semigroup $A{\kern-1pt}C_2$ is
$E$-separable.
\end{lemma}

\begin{proof}
This amounts to filling out the following table where for each
pair $p,q$ of distinct elements in the semigroup
$A{\kern-1pt}C_2$, we exhibit some idempotents $e$ and $f$ that
separate $p$ and $q$ respectively on the right and on the left.
\end{proof}

\begin{center}
\begin{tabular}{|c||c|c|c|c|c|c|c|c|c|}
\hline
$p$  & $c$        & 0                        & $a$  & $a$  & $a$  & $ab$ & $ab$ & $ba$\\
\hline
$q$  & $x\in A_2$ & $y\in A_2\setminus\{0\}$ & $b$  & $ab$ & $ba$ & $b$  & $ba$ & $b$\\
\hline
\hline
$e$  & 0          & $a$                      & $a$  & $ba$ & $a$  & $a$  & $a$  & $ba$\\
\hline
$pe$ & $c$        & 0                        & $a$  & $a$  & $a$  & $a$  & $a$  & $ba$\\
\hline
$qe$ & 0          & $a$ or $ba$             & $ba$ & 0    & $ba$ & $ba$ & $ba$ & 0\\
\hline
\hline
$f$  & 0          & $a$                      & $a$  & $a$  & $ab$ & $ab$ & $a$  & $ba$\\
\hline
$fp$ & $c$        & 0                        & $a$  & $a$  & $a$  & $ab$ & $ab$ & $ba$\\
\hline
$fq$ & 0          & $a$ or $ab$             & $ab$ & $ab$ &  0   & 0    & $a$  & $b$\\
\hline
\end{tabular}
\end{center}

The next result that we need is the union of the first part of
Proposition~3.2 in~\cite{LeeVolkov06} with the dual statement. By
$A_0$ we denote the subsemigroup $A_2\setminus\{a\}=\{b,ab,ba,0\}$
of the semigroup $A_2$.

\begin{lemma}
\label{connect} Let $S$ be an $E$-separable semigroup and
$A_0\in\var S$. Suppose that $S$ satisfies an identity $u=v$ such
that the word $u$ can be represented as $u_1u_2$ with
$\alf(u_1)\cap\alf(u_2) = \varnothing$. Then the word $v$ can be
represented as $v_1v_2$ such that $\alf(v_1) = \alf(u_1)$,
$\alf(v_2) = \alf(u_2)$ and the semigroup $S$ satisfies the
identities $u_1 = v_1$ and $u_2 = v_2$.
\end{lemma}

The next lemma is borrowed from~\cite{Ha97}, see Lemma~3.2 there.

\begin{lemma}
\label{kublanovsky} If for some $n\ge 1$ a semigroup $S$ satisfies
the identities
\begin{equation}
\label{RSn}
 x^2=x^{n+2},\ xyx=(xy)^{n+1}x ,\ xyx(zx)^n = x(zx)^nyx,
\end{equation}
then for every pair of distinct regular elements $p,q \in S$ there
exist a completely 0-simple semigroup $K$ and a surjective
homomorphism $\chi:S\rightarrow K$ such that $p\chi \ne q\chi$.
\end{lemma}

The last ingredient of our proof is a well-known result by
Houghton~\cite[Theorem~5.1]{Hough77} formulated in a convenient
for us way.
\begin{lemma}
\label{houghton} If the idempotents of a completely 0-simple
semigroup $S$ generate a combinatorial subsemigroup, then $S$ can
be presented as the Rees matrix semigroup $M^0(G; I, \Lambda; P)$
over a group $G$ such that every entry of the sand\-wich-matrix
$P$ is equal to either zero or the identity element of $G$.
\end{lemma}

\begin{proof1}
Recall that we have denoted by $\mathcal{V}$ the variety defined
by the identities \eqref{eq:1}--\eqref{eq:4}. By
Corollary~\ref{basis} we have the inclusion $\var
A{\kern-1pt}C_2\subseteq\mathcal{V}$. Arguing by contradiction,
assume that this inclusion is strict. Then there exists an
identity that holds in the semigroup $A{\kern-1pt}C_2$ but fails
in the variety $\mathcal{V}$. Among all such identities, we chose
an identity $u=v$ with the minimum possible number of variables in
the word $u$. We aim to show that the words $u$ and $v$ must be
connected.

Assume for the moment that, say, $u$ is not connected. This means
that it can be decomposed as $u=u_1u_2$ with
$\alf(u_1)\cap\alf(u_2)=\varnothing$. By Lemma~\ref{separability}
the semigroup $A{\kern-1pt}C_2$ is $E$-separable, and since
$A{\kern-1pt}C_2$ obviously contains the semigroup $A_0$ as a
subsemigroup, we see that Lemma~\ref{connect} applies to
$A{\kern-1pt}C_2$. By this lemma we have $v=v_1v_2$ where $\alf
(v_1) =\alf (u_1)$, $ \alf (v_2) = \alf (u_2)$ and both $u_1 =
v_1$ and $u_2 = v_2$ hold in the semigroup $A{\kern-1pt}C_2$.
Since $|\alf (u_1)|, |\alf (u_2)| < |\alf (u)|$, the choice of the
identity $u=v$ ensures that the identities $u_1=v_1$ and $u_2=v_2$
hold in the variety $\mathcal{V}$. Clearly, the identity $u=v$ is
a consequence of these two identities whence it also must hold in
$\mathcal{V}$, a contradiction. Analogously, one checks that the
word $v$ must be connected.

Now let $S$ be a semigroup in $\mathcal{V}$ such that the words
$u$ and $v$ take distinct values $p$ and $q$ under some
interpretation of variables. By Lemma~\ref{regular elements} these
values are regular elements. If we compare the identities that
define the  variety $\mathcal{V}$ with the three identities
\eqref{RSn} from the premise of Lemma~\ref{kublanovsky}, we see
that for $n=2$ the first two of the three identities coincide with
the identities \eqref{eq:1} and \eqref{eq:2} respectively while
the third one readily follows from the identity \eqref{eq:3}.
Thus, Lemma~\ref{kublanovsky} applies to the semigroup $S$ and its
regular elements $p$ and $q$. Therefore there exist a completely
0-simple semigroup $K$ and a surjective homomorphism
$\chi:S\rightarrow K$ such that $p\chi \ne q\chi$. Observe that
the elements $p\chi$ and $q\chi$ are also values of the words $u$
and $v$ under some interpretation of variables whence the identity
$u=v$ fails in the semigroup $K$. On the other hand, the semigroup
$K$ belongs to the variety $\mathcal{V}$ because it is a
homomorphic image of the semigroup $S\in\mathcal{V}$. This means
that we can use $K$ instead of $S$; in other words, we may (and
will) assume that the semigroup $S$ from the ``gap'' between the
varieties $\var A{\kern-1pt}C_2$ and $\mathcal{V}$ is completely
0-simple.

By Proposition~\ref{aperiodic core} the idempotents of $S$
generate a combinatorial subsemigroup, but then
Lemma~\ref{houghton} implies that $S$ can be presented as the Rees
matrix semigroup $M^0(G; I, \Lambda; P)$ over a group $G$ such
that every entry of the sandwich-matrix $P$ is equal to either
zero or the identity element of $G$. Let $T$ be the Rees matrix
semigroup  $M^0(E; I, \Lambda; P)$ over the trivial group
$E=\{1\}$ with the same sandwich-matrix $P$. It is known (see,
e.g., \cite[Proposition~1.2]{LeeVolkov06}) that every completely
0-simple semigroup over the trivial group belongs to the variety
generated by the semigroup $A_2$; in particular, $T\in\var A_2$.
Further, the group $G$ is isomorphic изоморфна to a maximal
subgroup in $S$ whence $G\in\mathcal{V}$. Therefore $G$ satisfies
the identity~\eqref{eq:1} and hence $G$ is a group of exponent~2.
It is well-known that every group of exponent~2 belongs to the
variety generated by the group $C_2$; in particular, $G\in\var
C_2$.

It is easy to verify that the mapping $T\times G\to S$ that sends
the pair $\bigl((i,1,\lambda),g\bigr)\in T\times G$ to the element
$(i,g,\lambda)\in S$ is a surjective homomorphism. Since $T\in\var
A_2$ and $G\in\var C_2$, we have
$$T\times G\in\var(A_2\times C_2)=\var A{\kern-1pt}C_2,$$
whence $S\in\var A{\kern-1pt}C_2$. This contradicts the choice of
the semigroup $S$. The theorem is proved.
\end{proof1}

\section{A polynomial algorithm for \VM{$A{\kern-1pt}C_2$}}

Given a semigroup $S$ with $|S|=n$, we want to test whether or not
$S$ belongs to the variety $\var A{\kern-1pt}C_2$. For this, by
Theorem~\ref{basis1}, it is necessary and sufficient to test
whether or not $S$ satisfies the identities
\eqref{eq:1}--\eqref{eq:4}. Testing the identities
\eqref{eq:1}--\eqref{eq:3} requires $O(n^3)$ time, see the
argument in the proof of Lemma~\ref{lemma 1.1}. No straightforward
test for the infinite identity series \eqref{eq:4} is possible but
here we can use the structural equivalent from
Proposition~\ref{aperiodic core}: it is necessary and sufficient
to test whether or not the subsemigroup of the semigroup $S$
generated by all idempotents of $S$ is combinatorial. We will show
that this can also be tested in $O(n^3)$ time.

Calculating squares of all elements of the semigroup $S$, we can
find the set of all idempotents in $S$ in $O(n)$ time. Let $T_1$
be this set and define inductively $T_{i+1}=T_iT_1$. It is clear
that constructing each set $T_{i+1}$ requires at most $n^2$ steps.
Further, it is easy to see that $T_i\subseteq T_{i+1}$ and that if
$T_k=T_{k+1}$ for some $k$, then $T_k=T_{k+\ell}$ for all $\ell$
whence $T_k$ is a subsemigroup in $S$. By the construction, every
element in $T_k$ is a product of idempotents, and therefore, $T_k$
coincides with the subsemigroup of the semigroup $S$ generated by
all idempotents of $S$. Since no strictly increasing chain of
subsets of $S$ can contain more than $n$ subsets, we have $k\le
n$, and the subsemigroup $T_k$ will be constructed this way in
$O(n^3)$ time. Now it remains to test whether or not $T_k$ is
combinatorial and for this it is necessary and sufficient to test
whether or not $T_k$ satisfies the identity \eqref{eq:5}, see the
proof of Proposition~\ref{aperiodic core}. This last check can be
done in $O(n)$ time.

Thus, we have proved the main result of the present paper:

\begin{thm}
The $6$-element semigroup $A{\kern-1pt}C_2$ has no finite identity
basis but, given a finite semigroup $S$, one can test the
membership of $S$ in the variety  $\var A{\kern-1pt}C_2$ in
$O(|S|^3)$ time.
\end{thm}

\medskip

\noindent\textbf{Acknowledgement.} The first and the second
authors acknowledge support from the Federal Education Agency of
Russia, project 2.1.1/3537, and from the Russian Foundation for
Basic Research, grants 09-01-12142 and 10-01-00524.

\end{document}